\documentclass[11pt]{amsart}
\usepackage{float}
\usepackage{graphicx}
\usepackage{amssymb}
\theoremstyle{plain}
\newtheorem{thm}{Theorem}[section]
\newtheorem{lemma}[thm]{Lemma}
\newtheorem{cor}[thm]{Corollary}

\theoremstyle{definition}
\newtheorem{defi}[thm]{Definition}

\newtheorem{rmk}[thm]{Remark}
\newtheorem{example}[thm]{Example}

\newcommand{\frb}{\mathfrak{b}}

\newcommand{\frg}{\mathfrak{g}}
\newcommand{\frh}{\mathfrak{h}}
\newcommand{\fri}{\mathfrak{i}}

\newcommand{\frn}{\mathfrak{n}}

\newcommand{\frp}{\mathfrak{p}}

\newcommand{\fru}{\mathfrak{u}}

\newcommand{\bbC}{\mathbb{C}}

\newcommand{\bbN}{\mathbb{N}}

\newcommand{\bbR}{\mathbb{R}}

\newcommand{\bbZ}{\mathbb{Z}}

\newcommand{\caA}{\mathcal{A}}

\newcommand{\caF}{\mathcal{F}}

\newcommand{\caL}{\mathcal{L}}

\begin{document}

\title{Ad-nilpotent ideals and The Shi arrangement}

\author{Chao-Ping Dong}

\address[Dong]{School of Mathematics,  Shandong University,
Jinan 250100, China}
\email{chaoping@sdu.edu.cn}

\abstract{We extend the Shi bijection from the Borel subalgebra case
to  parabolic subalgebras. In the process,   the $I$-deleted Shi arrangement $\texttt{Shi}(I)$ naturally emerges. This arrangement
interpolates between the Coxeter arrangement $\texttt{Cox}$ and the
Shi arrangement $\texttt{Shi}$, and breaks the symmetry of
$\texttt{Shi}$ in a certain symmetrical way. Among other things, we
determine the characteristic polynomial $\chi(\texttt{Shi}(I), t)$
of $\texttt{Shi}(I)$ explicitly for $A_{n-1}$ and $C_n$. More
generally, let $\texttt{Shi}(G)$ be an arbitrary arrangement between
$\texttt{Cox}$ and $\texttt{Shi}$. Armstrong and Rhoades recently gave a
formula for $\chi(\texttt{Shi}(G), t)$ for $A_{n-1}$.
Inspired by their result, we obtain formulae for
$\chi(\texttt{Shi}(G), t)$ for $B_n$, $C_n$ and $D_n$.}

\endabstract



\maketitle


\section{Introduction}
Let $\frg$ be a finite-dimensional complex simple Lie algebra of rank $l$. Fix a Cartan subalgebra $\frh$ of $\frg$. Then we have the root system $\Delta=\Delta(\frg, \frh)$. Let $V$ be the real vector space spanned by $\Delta$. We denote by $\langle\, ,\,\rangle$ the canonical inner product on $V$ which  is  induced from $\frh$ of the Killing form of $\frg$.
For convenience, we will equip $V$ with an inner product $(\, , \,)$ which is a  suitable scalar multiple of the canonical one. For any root $\alpha\in\Delta$, let $\frg_{\alpha}$ be the root space relative to $\alpha$. Let $\Pi=\{\alpha_1, \cdots, \alpha_l \}  \subseteq \Delta^{+}$ be a fixed choice of simple and positive root systems of $\Delta$, respectively. Let $\frn=\bigoplus_{\alpha\in\Delta^{+}}\frg_{\alpha}$. Then $\frb=\frh\oplus\frn$ is a Borel subalgebra of $\frg$.

The abelian ideals of a Borel subalgebra were studied by Kostant
\cite{K1, K2} in connection with the representation theory of
semisimple Lie groups. In particular, D.~Peterson's following theorem was
detailed in \cite{K2}: there are $2^{l}$ abelian ideals of $\frb$,
regardless of the type of $\frg$. Peterson's approach was to give a
bijection between the abelian ideals of $\frb$ and a certain set of
elements in the affine Weyl group $\widehat{W}$ of $\frg$. This surprising result led Cellini and Papi to find
similarities for \emph{ad-nilpotent} ideals of $\frb$, i.e., the
ideals of $\frb$ which are contained in $\frn$.  For example, they showed how to
associate to any ad-nilpotent ideal $\fri$ of $\frb$ a uniquely
determined element $w_{\fri}\in \widehat{W}$ in \cite{CP1}. In
\cite{CP2}, they gave further a bijection between the set of
ad-nilpotent ideals of $\frb$ and the set of $W$-orbits in
$\check{Q}/(h+1)\check{Q}$, where $\check{Q}$ is the coroot lattice and $h$ is the Coxeter number.

On the other hand, we note that  a bijection between the
set of all ad-nilpotent ideals of $\frb$ and the dominant regions of
the now-called Shi arrangement had been given by  Shi in \cite{S3}. Thus, Theorems 3.2 and 3.6 there count the number of all ad-nilpotent ideals of $\frb$.  To state his result,
let us recall some notations concerning hyperplane arrangements.

A \emph{hyperplane arrangement} is a finite collection of affine hyperplanes in an Euclidean space. For example, the \emph{Coxeter arrangement}
  associated with $\Delta^{+}$ is the arrangement in $V$ defined by
\begin{equation}\label{Coxeter-arrangement}
\texttt{Cox}:=\{ H_{\alpha, 0}\mid \alpha\in\Delta^{+}\}.
\end{equation}
Here for $\alpha\in\Delta^{+}$ and $k\in\bbZ$, we define a hyperplane
\begin{equation}\label{H-alpha-k}
H_{\alpha, k}:=\{ v\in V \mid (v, \alpha)=k  \}.
\end{equation}
If $\mathcal{A}$ is a hyperplane arrangement in $V$, the connected components of
$V - \bigcup_{H\in \mathcal{A}}H$ are called \emph{regions}. For example, there are $|W|$ regions of \texttt{Cox}, where $W$ is the Weyl group associated to $\Delta(\frg, \frh)$. For later use, let us single out the \emph{dominant region} of \texttt{Cox} as follows:
\begin{equation}\label{Cox-dominant-region}
V_{\infty}:=\{ v\in V \mid (v, \alpha)>0, \forall\alpha\in\Delta^{+} \}.
\end{equation}

By the idea of Postnikov and Stanley in \cite{PS}, a \emph{deformation} of the Coxeter arrangement  is an affine arrangement each of whose hyperplanes is parallel to some one in $\texttt{Cox}$.
The \emph{Shi arrangement} \texttt{Shi} associated to $\Delta^{+}$ can be viewed as such an example:
\begin{equation}\label{Shi-arrangement}
\texttt{Shi}:=\texttt{Cox} \cup \{ H_{\alpha, 1}\mid \alpha\in \Delta^{+}\}.
\end{equation}
This arrangement was defined by Shi
in the study of the Kazhdan-Lusztig cellular structure of the affine
Weyl group of type $A$, see Chapter 7 of \cite{S1}. A region of \texttt{Shi} is called
\emph{dominant} if it is contained in $V_{\infty}$.
For any $\frh$-stable subset  $\fru$ of $\frg$, let $\Phi_{\fru}\subset \Delta$ be the subset defined so that
$$\fru=\fru\cap\frh + \sum_{\alpha\in\Phi_{\fru}} \frg_{\alpha}.
$$
Now let us cite the Shi bijection from Theorem 1.4 of \cite{S3} as
follows:

\begin{thm} \emph{(\textbf{Shi})}\label{thm-Shi}
There exists a natural bijective map from the set of all the ad-nilpotent
ideals of $\frb$ to the set of all the dominant regions of the
hyperplane arrangement \emph{\texttt{Shi}}. The map sends $\fri$ to
$\{ v\in V_{\infty} \mid  (v, \beta) >1, \forall \beta\in
\Phi_{\fri};  (v, \beta) <1, \forall \beta\in
\Delta^{+}\setminus\Phi_{\fri} \}$.
\end{thm}

The first purpose of this paper is to generalize the Shi bijection from the Borel subalgebra case to parabolic subalgebras, see Theorem \ref{thm-main}.  In the process, the $I$-deleted Shi arrangement $\texttt{Shi}(I)$ naturally emerges,  where $I$ is an arbitrary subset of $\Pi$. This arrangement interpolates between $\texttt{Cox}$ and $\texttt{Shi}$, see (\ref{I-deleted-Shi-arrangement}).  More generally, let us consider
\begin{equation}\label{Shi-G}
\texttt{Shi}(G)=\texttt{Cox}\cup\{H_{\alpha, 1}\mid \alpha\in G\},
\end{equation}
where $G$ is any subset of $\Delta^{+}$.

Recall that the fundamental combinatorial object associated with  a hyperplane arrangement $\caA$ in $V$ is its \emph{intersection poset} $L(\caA)$,
 which is defined as the set of nonempty intersections of hyperplanes from $\caA$, partially ordered by the \emph{reverse} inclusion of subspaces. As an invariant distilled from $\caL(\caA)$, the \emph{characteristic polynomial} $\chi(\caA, t)\in\bbZ[t]$ of  $\caA$ is defined by
\begin{equation}\label{def-char-poly}
\chi(\caA, t)=\sum_{x\in L(\caA)} \mu(V, x)  t^{\text{dim} (x)},
\end{equation}
where $\mu: L(\caA)\times L(\caA)\to \bbZ$ is the M\"obius function of the poset $L(\caA)$, see (3.15) and section 3.11 of \cite{St}.

When $\frg=A_{n-1}$, by the finite field method of Crapo and Rota \cite{CR},  Armstrong and Rhoades recently gave a formula for $\chi(\texttt{Shi}(G), t)$ in \cite{AR}. See Theorem \ref{thm-AR}. Analyzing its proof, one sees that there are two key features: for a large prime $p$, let $f(S)$ be the number of vectors in $\mathbb{F}_p^l -\texttt{Cox}_p$ satisfying $(H_{\alpha, 1})_p$ for all $\alpha\in S$. Then the first feature is that $\chi(\texttt{Shi}(G), p)$ can be expressed as an alternating sum
\begin{equation}\label{alt-sum}
\sum_{S\subseteq G} (-1)^{|S|} f(S),
\end{equation}
and  the summation is reduced to certain subsets of $G$. The second one is that $f(S)$   is shown to be dependent only on $|S|$.

We find a uniform way to express the first feature. Indeed, as
recorded in Lemma \ref{lemma-sum-quasi-antichain}, regardless of the type of $\frg$, one can
always express $\chi(\texttt{Shi}(G), p)$ as  (\ref{alt-sum}), and
it suffices to take the summation over the \emph{quasi-antichains of $\Delta^{+}$} (cf. Definition \ref{def-quasi-antichain}) contained in $G$. A detailed study of quasi-antichains will be given in Section 3. In particular, we will show that  they are in bijection with  the elements of $\caL(\texttt{Cox})$. When $\frg=A_{n-1}$, $B_n$, $C_n$ or $D_n$, we put
\begin{equation}\label{Stir-uniform}
\texttt{Stir}(G, k) := \# \Big\{ \mbox{quasi-antichains of }
\Delta^{+} \mbox{ contained in } G \mbox{ with size } n-k\Big\}.
\end{equation}
In particular if $\frg=A_{n-1}$, as we shall see in (\ref{Stir-A}), the definition (\ref{Stir-uniform}) agrees with the original terminology $\texttt{Stir}(G, k)$ in \cite{AR}, and $\texttt{Stir}(\Delta^+, n-k)$ is nothing but  the number of partitions
of $[n]:=\{1, 2, \cdots, n\}$ into $k$ blocks.  The latter is usually denoted by $S(n, k)$, and termed as  the
\emph{Stirling number of the second kind}.  Recall from (1.93)
and (1.94d) of \cite{St} that we have the recurrence
\begin{equation}\label{stir-recur-A}
S(n, k)=k S(n-1, k)+ S(n-1, k-1),
\end{equation}
and the combinatorial identity
\begin{equation}\label{Gamman-identity} \sum_{k=0}^{n} S(n,
k)(t)_k=t^{n}.
\end{equation}
Here $(t)_k=t(t-1)\cdots(t-k+1)$ is the falling factorial. It is  interpreted  as $1$ whenever $k\leq 0$, and  will be used throughout this paper.  We will give some analogs of
(\ref{stir-recur-A}) and (\ref{Gamman-identity}) in Lemma \ref{lemma-Tn-Lambdan}.

By explicit calculations, we find that the second feature still holds
when $\frg$ is $B_n$ or $C_n$. This leads us to

\begin{thm}
\emph{(Theorem \ref{thm-ShiG-type-C} and Remark \ref{rmk-type-B})}
\label{thm-ShiG-type-BC}
Let $\frg$ be $B_n$ or $C_n$, $n\geq 2$.
For any subset $G\subseteq \Delta^+$, the characteristic polynomial of $\emph{\texttt{Shi}}(G)$ is given by
$$
\chi(\emph{\texttt{Shi}}(G), t)=\sum_{k=0}^{n} (-1)^k \emph{\texttt{Stir}}(G, n-k)2^{n-k}(\frac{t-1}{2}-k)_{n-k}.
$$
\end{thm}

 When $\frg=D_n$, $f(S)$ no longer  depends only on $|S|$.  However, by
a more careful analysis, we still obtain a formula for
$\chi(\texttt{Shi}(G), p)$, see Theorem
\ref{thm-ShiG-type-D}.
Our explicit calculations also show that when $\frg$ is classical and $p$ is large enough,  $f(S)$ is nonzero if and only if $S$ is a quasi-antichain of $\Delta^+$, see Remark \ref{rmk-quasi-antichain}. Thus the reduction of the alternating sum (\ref{alt-sum}) to quasi-antichains  turns out to be precise.

Now let us specialize the general story to the $I$-deleted Shi arrangement $\texttt{Shi}(I)$. Going from $\texttt{Shi}$ to $\texttt{Shi}(I)$, the symmetry is broken kind of symmetrically. Thus we may expect $\texttt{Shi}(I)$ to behave better than an arbitrary $\texttt{Shi}(G)$.
Indeed, when $\frg=A_{n-1}$, we find that: the polynomial $\chi(\texttt{Shi}(I), t)$ factors into nonnegative integers and it depends only on $|I|$; moreover, the cone of \texttt{Shi}$(I)$ is free in the sense of Terao \cite{Te}, see Theorem \ref{thm-type-A}. These results are based on the works of Athanasiadis \cite{At1, At2}, Armstrong and Rhoades \cite{AR}. When $\frg=C_n$, we find that: $\chi(\texttt{Shi}(I), t)$ always factors into nonnegative integers; moreover, if $I$ contains $2e_n$, $\chi(\texttt{Shi}(I), t)$ depends only on $|I|$, and the same conclusion holds if $I$ does not contain $2e_n$, see Theorem \ref{thm-type-C}. These results are obtained via Theorem \ref{thm-ShiG-type-BC} and Lemma \ref{lemma-Tn-Lambdan}.

The paper is organized as follows: we generalize the Shi bijection in Section 2.
We collect some preliminaries on the characteristic polynomial of a hyperplane arrangement in Section 3.
In particular, the finite field method is described there. Moreover, we introduce the quasi-antichains and give a detailed study of them.
Section 4 is devoted to the study of \texttt{Shi}($I$) for $A_{n-1}$, while Section 5 handles the $C_n$ case.
 Finally, a formula for $\chi(\texttt{Shi}(G), t)$ is deduced for $D_n$ in Section 6.

\section{A generalization of the Shi bijection}

This section is devoted to  generalizing the Shi bijection (see Theorem \ref{thm-Shi}) from the Borel subalgebra case to parabolic subalgebras. Let us begin with
some preliminaries. We endow $\Delta^{+}$ with the usual
partial ordering. Namely, $\alpha\leq\beta$ if
$\beta-\alpha=\sum_{\gamma\in\Delta^{+}} c_{\gamma}\gamma$, where
the $c_{\gamma}$ are some non-negative real numbers. Any subset $P$ of $\Delta^{+}$ inherits a
partial ordering from $(\Delta^{+}, \leq)$. Let us denote the
corresponding poset by $(P, \leq)$ or simply by $P$.
 Recall that a \emph{dual order ideal} of $P$ is a subset $J$ of $P$ such that if $t\in J$ and $t\leq s$ for $s\in P$, then $s\in J$. Recall also that an \emph{antichain} of $P$
 is a subset of $P$ consisting of pairwise non-comparable elements.
 Note that there is a canonical bijection  between the dual order ideals of $P$ and the antichains of $P$: given a dual order ideal,  we send it to the set of its minimal elements; the inverse map sends the antichain $\{a_1, \cdots, a_k\}$ to the dual order ideal which is the union of the principal dual order ideals $V_{a_1}, \cdots, V_{a_k}$, where $V_{a}=\{b\in P\mid a\leq b\}$.

Fix a subset $I\subseteq \Pi$.
Let $\Delta_I$ be the sub root system of $\Delta$ spanned by $I$, and put $\Delta^{+}_I=\Delta_I\cap \Delta^{+}$. Let
$$
\frp_I=\frh+ \sum_{\alpha\in \Delta_I\cup \Delta^{+}} \frg_{\alpha}
$$
be the standard parabolic subalgebra of $\frg$ corresponding to $I$.
Recall that an ideal $\mathfrak{i}$ of $\frp_{I}$ is called \emph{ad-nilpotent} if for all $x\in\fri$,
$\text{ad}_{\frp_I} x$ is nilpotent. Let
 \begin{equation}\label{C-I}
 C_I=\{\beta\in \Delta^{+}\setminus \Delta_I \mid \forall \alpha\in \Delta_{I}^{+},  \beta-\alpha\notin \Delta \}.
  \end{equation}
We define the \emph{$I$-deleted Shi arrangement} as
\begin{equation}\label{I-deleted-Shi-arrangement}
\texttt{Shi}(I):=\texttt{Cox} \cup \{H_{\alpha, 1}\mid \alpha\in C_I\}.
\end{equation}
Since \texttt{Shi}($\Pi$) is \texttt{Cox} and \texttt{Shi}($\emptyset$) is \texttt{Shi}, we see that \texttt{Shi}($I$) interpolates between the Coxeter arrangement and the Shi arrangement.
Again a region of \texttt{Shi}($I$) is called dominant if it is
contained in $V_{\infty}$.  For any
ad-nilpotent ideal $\fri$ of $\frp_I$, we put $
\Psi_{\fri}=\cup_{\alpha} \{\beta\in C_I\mid  \alpha \leq \beta \}$,
where $\alpha$ runs over all the minimal elements of $(\Phi_{\fri},
\leq)$. Then we have

\begin{thm}\label{thm-main} There exists a
natural bijective map from the set of all the ad-nilpotent ideals of
$\frp_I$ to the set of all the dominant regions of the hyperplane
arrangement \emph{\texttt{Shi}}($I$). This map sends $\fri$ to $\{
v\in V_{\infty} \mid (v, \beta) >1, \forall \beta\in \Psi_{\fri};
(v, \beta) <1, \forall \beta\in C_I \setminus\Psi_{\fri} \}$.
\end{thm}

We note that the number of the ad-nilpotent ideals for $\frp_I$ was enumerated by Righi in Theorem 5.12 of \cite{R1} for $\frg$
classical, and in \cite{R2} for $\frg$ exceptional using GAP4. As a
consequence, we also know the number of dominant regions of
$\texttt{Shi}(I)$. Theorem \ref{thm-main} will be proved by collecting the bijections in the following subsections.

\subsection{} Recall that a subset $I$ of $\Pi$ is fixed. Let us put
$$\caF_I :=\{ \Phi\subseteq\Delta^{+}\setminus \Delta_I: \mbox{ if } \alpha\in\Phi, \beta\in\Delta^{+}\cup\Delta_I \mbox{ and } \alpha+\beta\in\Delta^{+}, \mbox{ then } \alpha+\beta\in \Phi \}.
$$
As noted in section 3 of \cite{R1}, we have a bijection
\begin{equation}\label{map-1}
\{\mbox{ad-nilpotent ideals of } \frp_I\}\to \caF_I; \fri\mapsto \Phi_{\fri}.
\end{equation}

\subsection{} For any $\Phi\in \caF_I$, let $A(\Phi)$ be the set of all the minimal elements of $(\Phi, \leq)$. We note that $A(\Phi)$ is contained in $C_I$. Indeed, let $\beta$ be any minimal element of $(\Phi, \leq)$, and take any $\alpha\in \Delta^{+}_I$, it suffices to show that $\beta-\alpha$ is not a root. Assuming  the contrary gives
$\beta+(-\alpha)\in \Delta^{+}\setminus \Delta_I$.
Since $\beta\in\Phi$, $-\alpha\in \Delta^{+}\cup\Delta_I$ and $\Phi\in\caF_I$, we conclude that $\beta-\alpha\in\Phi$, which contradicts to the minimality of $\beta$.
Thus  $A(\Phi)$ is an antichain of $(C_I, \leq)$ and we have a well-defined  map
\begin{equation}\label{map-2}
\caF_I\to \{\mbox{antichains of } C_I\}; \Phi\mapsto A(\Phi).
\end{equation}
Actually, the above map is bijective. To show this, it suffices to prove that for any antichain $A$ of $C_I$, the set
$$
 \Phi(A)=\bigcup_{\beta\in A}\{\alpha\in \Delta^+: \alpha\geq\beta\}
$$
belongs to $\caF_I$, which is precisely
the content of Proposition 1.4 of \cite{CDR} in view of the following

\begin{lemma}\label{lemma-CI} We have $C_I=\{\beta\in \Delta^{+}\setminus \Delta_I \mid \forall \alpha\in I,  \beta-\alpha\notin \Delta \}$.
\end{lemma}

\begin{proof}
  Fix any $\beta\in\Delta^+\setminus \Delta_I$ such that $\beta-\alpha$ is not a root for any $\alpha\in I$, it suffices to show that $\beta-\gamma$ is not a root for any $\gamma\in\Delta_I^{+}$. Let us prove this by induction on the height of $\gamma$. There is nothing to prove when $\mbox{ht}(\gamma)=1$. Suppose that we have proved it for all $\gamma^{\prime}\in \Delta_I^+$ with $\mbox{ht}(\gamma^{\prime})<r$. Now take  $\gamma\in \Delta_I^+$ be such that $\mbox{ht}(\gamma)=r$. Choose $\alpha\in I$ such that $(\gamma, \alpha)>0$. Then $\gamma -\alpha\in \Delta_I^+$. By assumption, $\beta-\alpha$ is not a root, thus  $(\beta, \alpha)\leq 0$ and
$(\beta-\gamma, \alpha) <0$. Thus $\beta-\gamma$ is not a root since otherwise $\beta-(\gamma-\alpha)$ would be a root, contradicting to the induction hypothesis since  $\mbox{ht}(\gamma-\alpha)=r-1$.
\end{proof}

\subsection{} We take the canonical bijection from $\{ \text{antichains of } C_I \}$ to $\{ \text{dual order ideals of } C_I \}$.

\subsection{}
Given any dual order ideal $\Phi$ of $C_I$,  define
$$R_{\Phi}=\{ v\in V_{\infty} \mid (v, \beta) >1, \forall \beta\in \Phi;  (v, \beta) <1, \forall \beta\in C_I\setminus\Phi \}.$$ Then the map
\begin{equation}\label{map-4}
\{ \mbox{dual order ideals of } C_I \}\to \{ \mbox{dominant regions of } \texttt{Shi}(I)   \}; \Phi\mapsto R_{\Phi}.
\end{equation}
is well-defined in view of the following

\begin{lemma} The set $R_{\Phi}$ is non-empty. Hence it is a dominant region of $\emph{\texttt{Shi}}(I)$.
\end{lemma}
\begin{proof}
Let $\widetilde{\Phi}$ be the unique dual order ideal of $\Delta^{+}$ containing $\Phi$, that is,
$$
\widetilde{\Phi}=\bigcup_{\beta\in \Phi}\{\alpha\in \Delta^+: \alpha\geq\beta\}.
$$ We claim that
$C_I\setminus \Phi \subseteq \Delta^{+}\setminus \widetilde{\Phi}$.
Indeed, if there exists $\alpha\in C_I\setminus \Phi$ such that $\alpha\in \widetilde{\Phi}$, then $\alpha=\beta+\gamma$ for some $\beta\in\Phi$ and $\gamma\in\Delta^{+}$. Since $\alpha\in C_I$ and $\Phi$ is a dual order ideal of $C_I$, this would imply that $\alpha\in\Phi$, which is absurd. Thus the claim follows, and we have
$R_{\widetilde{\Phi}}\subseteq R_{\Phi}$, where
$R_{\widetilde{\Phi}}=\{ v\in V_{\infty} \mid (v, \beta) >1, \forall \beta\in \widetilde{\Phi};  (v, \beta) <1, \forall \beta\in \Delta^{+}\setminus\widetilde{\Phi} \}$.
Note that $R_{\widetilde{\Phi}}$ is non-empty by Theorem \ref{thm-Shi}. Thus, $R_{\Phi}$ is non-empty as well. Then it is immediate that $R_{\Phi}$ is a dominant region of $\texttt{Shi}(I)$.
\end{proof}

Given any dominant region $R$ of $\texttt{Shi}(I)$, we say that $R$ is \emph{above} the hyperplane $H_{\beta, 1}$, $\beta\in\Delta^{+}$, if it contains $R$ and the origin of $V$  in different half spaces. Now define
$$\tau(R)=\{ \beta\in C_I \mid R \text{ is above the hyperplane } H_{\beta, 1}  \}.$$
We note that $\tau(R)$ is a dual order ideal of $C_I$. Indeed, take any $\gamma\in C_I$ such that $\beta\leq\gamma$ for some $\beta\in\tau(R)$,  then since $R$ is contained in $V_{\infty}$, we have
$$
(v, \gamma)=(v, \beta) + (v, \gamma-\beta) \geq (v, \beta)>1, \forall v\in R.
$$
Therefore, $\gamma\in\tau(R)$ as desired. Now it is obvious that the map (\ref{map-4}) is bijective with the inverse given by $\tau$.

\section{Characteristic polynomial, the finite field method and quasi-antichains}

Let $\caA$ be a collection of hyperplanes in the Euclidean space $\bbR^n$. In this section, we will collect some preliminaries concerning the determination of $\chi(\caA, t)$. In particular, the finite field method due to Crapo and Rota \cite{CR} will be described. Moreover, we will introduce quasi-antichains, and give a detailed study of them.

\subsection{Characteristic polynomial and Poincar\'e polynomial}

The \emph{Poincar\'e polynomial} of $\caA$ is defined by
$$
P(\caA, t)=\sum_{x\in L(\caA)} \mu(V, x)  (-t)^{n- \text{dim} (x)}.
$$
Comparing it with (\ref{def-char-poly}), one sees easily that
$$
P(\caA, t)=(-t)^n \chi(\caA, -\frac{1}{t}) \mbox{ and }
\chi(\caA, t)=t^n P(\caA, -\frac{1}{t}).
$$

\begin{example} Consider $G_2$. Let $\Pi=\{\alpha_1, \alpha_2\}$, where $\alpha_1$ is the short simple root and $\alpha_2$ is the long simple root. Then one can easily compute from definition that
\begin{itemize}
\item[$\bullet$] If $I=\{\alpha_1\}$, then $\chi(\texttt{Shi}(I), t)=(t-3)(t-5)$;
\item[$\bullet$] If $I=\{\alpha_2\}$, then $\chi(\texttt{Shi}(I), t)=(t-4)(t-5)$.
\end{itemize}
\end{example}
\begin{rmk}
Since the negative of the coefficient of $t$ in $\chi(\texttt{Shi}(I), t)$ is $|C_I|+|\Delta^{+}|$, in general, we shall not expect $\chi(\texttt{Shi}(I), t)$ to depend only on the cardinality of $I$.
\end{rmk}

Suppose that the normals to the hyperplanes of $\caA$ span a subspace of $V\subseteq \bbR^n$ of dimension $r(\caA)$. We call $r(\caA)$ the \emph{rank} of $\caA$. We say a region of $\caA$ is \emph{relatively bounded} if its intersection with $V$ is bounded.
When $V=\bbR^n$, a region of $\caA$ is relatively bounded if and only if it is bounded.
A good reason to study the characteristic polynomial for hyperplane arrangements is given by the following classic theorem of Zaslavsky.

\begin{thm}\label{thm-Za} \emph{(Section 2 of \cite{Za})} For any hyperplane in $\bbR^n$, we have
\begin{itemize}
\item[$\bullet$] The number of regions of $\caA$ is $(-1)^n\chi(\caA, -1)$;
\item[$\bullet$]
 The number of relatively bounded regions of $\caA$ is $(-1)^{r(\caA)} \chi(\caA, 1)$.
\end{itemize}
\end{thm}

Based on Shi's determination of the number of regions of $\texttt{Shi}$ \cite{S1, S2}, Headley
obtained the characteristic polynomial of  $\texttt{Shi}$.
\begin{thm}\label{thm-HeadAt} \emph{(Theorem 2.4 of \cite{Hea})} Let $\frg$ be a finite-dimensional complex simple Lie algebra with rank $l$ and Coxeter number $h$. Let $V$ be the real vector space spanned by $\Delta(\frg, \frh)$. Then for the Shi arrangement in $V$, we have
$$
P(\emph{\texttt{Shi}}, t)=(1+ht)^l.
$$
\end{thm}

Recall that $h$ is the Coxeter number, which equals to $n+1, 2n, 2n$ and $2(n-1)$ for $A_{n}$, $B_n$, $C_n$ and $D_n$ respectively, see Section 3.18 of \cite{Hum}. Let us denote the multiset of the roots of $\chi(\caA, t)$ by $\exp(\caA)$.  For example, $\exp(\texttt{Shi})=\{h^l \}$ by the above theorem. Here and in what follows we write $\{ a_1^{m_1}, a_2^{m_2}, \cdots, a_r^{m_r}\}$ for a multiset, where each $m_i$ stands for the multiplicity of $a_i$. The Factorization Theorem of Terao \cite{Te} states that the characteristic polynomial $\chi(\caA, t)$ factors over the nonnegative integers for any free arrangement $\caA$.

\subsection{The finite field method}
Athanasiadis \cite{At1} offered a different approach to Theorem \ref{thm-HeadAt} which did not rely on Shi's result. There the main tool was the \emph{finite field method} of Crapo and Rota \cite{CR} which turned out to be very useful. Let us describe this method. Suppose that the defining equations for the hyperplanes in $\caA$ have coefficients in $\bbZ$. Let $p\in\bbZ$ be a prime number and consider a hyperplane $H\in\caA$ with defining equation $a_1x_1+\cdots+a_nx_n=b$, where $a_i, b\in\bbZ$. Then we define the following subset $H_p$ of the finite vector space $\mathbb{F}_p^n$ by reducing the coefficients of $H$ modulo $p$:
$$
H_p:=\{(x_1, \cdots, x_n)\in\mathbb{F}_p^n \mid a_1x_1+\cdots+a_nx_n=b\}.
$$
We note that when $p$ is large enough, each $H_p$ is a hyperplane in $\mathbb{F}_p^n$, and we call $\caA_p:=\{H_p\mid H\in\caA\}$ the \emph{reduced hyperplane arrangement} of $\caA$.

\begin{thm}\label{thm-finite-field} \emph{(\cite{CR})} Let $p\in \bbZ$ be a large prime, and let $\caA$ be a collection of hyperplanes in $\bbR^n$ whose hyperplanes have defining equations with coefficients in $\bbZ$. Then the characteristic polynomial of $\caA$ satisfies
$$
\chi(\caA, p)=\# \Big(\mathbb{F}_{p}^{n}-\bigcup_{H\in\caA} H_p\Big).
$$
That is, $\chi(\caA, p)$ counts the number of points in the complement of the reduced arrangement $\caA_p$ in the finite vector space $\mathbb{F}_p^n$.
\end{thm}

A proof of the above theorem may also be found on pages 199--200 of \cite{At1}.

\subsection{Expressing $\chi(\texttt{Shi}(G), p)$ as an alternating sum over the quasi-antichains}
When $\frg$ is a finite-dimensional simple Lie algebra over $\bbC$ with rank $l$, the arrangements $\texttt{Shi}$ and $\texttt{Cox}$, thus any arrangement between them, can be realized as arrangements whose hyperplanes  have defining equations in $\bbZ$. Let $G$ be any subset of $\Delta^{+}$.
Recall the arrangement $\texttt{Shi}(G)$ defined in  (\ref{Shi-G}).
We shall prepare a lemma for the computation of its characteristic polynomial.

\begin{defi}\label{def-quasi-antichain}
We call a subset $S\subseteq \Delta^{+}$ a \emph{quasi-antichain} if for any two distinct elements $\alpha, \beta$ in S, the difference $\alpha - \beta$ is not a nonzero integer multiple of any root. We call the cardinality of a quasi-antichain its \emph{size}.
\end{defi}

Of course, any antichain of $(\Delta^{+}, \leq)$ is automatically a quasi-antichain of $\Delta^{+}$. The notion of quasi-antichain will be illustrated vividly in the following sections. It is motivated by the following

\begin{lemma}\label{lemma-sum-quasi-antichain}
Let $p$ be a large prime.
We have
$$
\chi(\emph{\texttt{Shi}}(G), p)=\sum_{\scriptsize S \subseteq G
\mbox{ is a quasi}\atop \mbox{-antichain of } \Delta^{+} }
(-1)^{|S|} f(S),
$$
where  $f(S)$ is the number of vectors  in $\mathbb{F}_p^l -\emph{\texttt{Cox}}_p$  satisfying $(H_{\alpha, 1})_p$ for all $\alpha\in S$.
\end{lemma}
\begin{proof}
By Theorem \ref{thm-finite-field}, it suffices to count the number of vectors in $\mathbb{F}_p^l - \texttt{Shi}(G)_p$.
By the principle of inclusion-exclusion (see for example Chapter 2 of \cite{St}), we have
$$
\chi(\texttt{Shi}(G), p)=\sum_{\scriptsize S \subseteq G} (-1)^{|S|}
f(S).
$$
The lemma follows from the observation that $f(S)=0$ if $S\subseteq G$ is not a quasi-antichain of $\Delta^+$. Indeed, in such a case, we can find two distinct roots $\alpha$ and $\beta$ in $S$ such that $\alpha-\beta=k\gamma$, where $k$ is a positive integer and $\gamma\in\Delta^{+}$. Now if there exists $v$ in  $\mathbb{F}_p^l -\texttt{Cox}_p$ satisfying the equations
$(H_{\alpha, 1})_p$ and $(H_{\beta, 1})_p$,
it would satisfy $(H_{k\gamma, 0})_p$ as well. Since $p$ is large enough, it would be a solution of
$(H_{\gamma, 0})_p$,
contradicting to the assumption that $v\notin\texttt{Cox}_p$.
\end{proof}

\begin{rmk}\label{rmk-quasi-antichain}
As we shall see in Corollaries \ref{Cor-type-A}, \ref{Cor-type-C} and \ref{Cor-type-D}, when $\frg$ is classical and $p$ is large enough,  $f(S)$ is nonzero if and only if $S$ is a quasi-antichain of $\Delta^+$.
\end{rmk}

\begin{lemma}\label{lemma-quasi-antichain-linear-independent}
Any quasi-antichain  $S=\{\beta_1, \cdots, \beta_k\}\subseteq \Delta^+$ is linearly independent. In particular,   $0\leq k\leq l$.
\end{lemma}
\begin{proof}
Suppose that $k>0$. Observe that $(\beta_i, \beta_j)\leq 0$ for any $i\neq j$. Indeed, otherwise $\beta_i-\beta_j$ would be a root, contradicting to the assumption that $S$ is a quasi-antichain. Now the same argument of page 9 of \cite{Hum} verifies the linear independence of $S$.
\end{proof}

\subsection{Quasi-antichains  and elements of $\caL(\texttt{Cox})$}

\begin{defi}\label{def-sub-closed}
We say that a subset $S\subseteq \Delta^{+}$ is \emph{sub-closed} if  $\alpha - \beta=k\gamma$, where $\alpha, \beta\in S$, $k\in\bbN^+$ and $\gamma\in\Delta^+$,  implies that $\gamma\in S$. Here ``sub" stands for substraction.
\end{defi}

It is obvious that any quasi-antichain of $\Delta^+$ is sub-closed.
For any
subset $S\subseteq \Delta^{+}$, we put $H_S:=\bigcap_{\alpha\in S}H_{\alpha, 0}$, and make the convention that $H_{\emptyset}=V$.
Moreover, we denote the intersection of all the sub-closed subsets of  $\Delta^{+}$ containing $S$ by $(S)$. One sees easily that $(S)$ is the smallest
sub-closed subset of  $\Delta^{+}$ containing $S$.

\begin{lemma}\label{lemma-S-gen-S}
For any
subset $S\subseteq \Delta^{+}$, we have $H_{(S)}=H_S$.
\end{lemma}
\begin{proof}
The desired conclusion follows from the observation that if  $\alpha - \beta=k\gamma$, where $\alpha, \beta, \gamma\in\Delta^+$, and $k\in\bbN^+$, then $$
H_{\alpha, 0}\cap H_{\beta, 0}\cap H_{\gamma, 0}=H_{\alpha, 0}\cap H_{\beta, 0}.$$
\end{proof}

\begin{lemma}\label{lemma-quasi-antichain-Cox}
For any sub-closed subset $S\subseteq \Delta^+$, there exists a unique quasi-antichain $S^{\prime}\subseteq S$ such that $H_S=H_{S^{\prime}}$.
\end{lemma}
\begin{proof}
Existence: let us do induction on $|S|$. When $|S|= 0$ or $1$, there is nothing to prove. Suppose that $S^{\prime}$ exists when $|S|\leq k$. Now let us consider the case that $|S|= k+1$ and $S$ is \emph{not} a quasi-antichain. By Definition \ref{def-quasi-antichain},
 $$
 A:=\{\alpha\in S|  \exists \beta\in S \text{ s.t. }
 \alpha-\beta=k\gamma
 \text{ for some } \gamma\in\Delta^+ \text{ and } k\in\bbN^+
  \}
 $$
 is non-empty. Pick up a maximal element $\alpha_0$ of $(A, \leq)$. Then there exists $\beta_0\in S$ such that
 $\alpha_0-\beta_0=k_0\gamma_0$,
for some $\gamma_0\in\Delta^+$  and $k_0\in\bbN^+$. Note that $\gamma_0\in S$ since $S$ is sub-closed. We claim that $S_1:=S \setminus \{\alpha_0\}$ is still sub-closed. Indeed, for any $\alpha, \beta\in S_1$ such that
$\alpha - \beta=k\gamma$, where  $\gamma\in\Delta^+$ and  $k\in\bbN^+$, we have $\gamma\in S$ since $S$ is sub-closed. Notice that $\gamma\neq \alpha_0$. Otherwise, we would have $\alpha\in A$, $\alpha_0\leq \alpha$ and $\alpha\neq \alpha_0$, contradicting to our choice of $\alpha_0$. Thus $\gamma\in S_1$ and the claim holds. Since $|S_1|=k$, by induction hypothesis, there exists a quasi-antichain $S^{\prime}\subseteq S_1$ such that $H_{S^{\prime}}=H_{S_1}$. Since
$$
H_{\alpha_0, 0}\cap H_{\beta_0, 0}\cap H_{\gamma_0, 0}=H_{\beta_0, 0}\cap H_{\gamma_0, 0},
$$
we conclude that $H_S=H_{S_1}=H_{S^{\prime}}$.

Uniqueness: reviewing the previous paragraph, we see that
every element of $S\setminus S^{\prime}$ can be expressed as a linear combination with  coefficients in $\bbN^+$  of two or more elements of $S^{\prime}$. Moreover, $S^{\prime}$ is linearly independent by Lemma \ref{lemma-quasi-antichain-linear-independent}.
Thus $S^{\prime}$ can be \emph{characterized} as the set of all roots $\alpha\in S$ such that $\alpha$ is not expressible as a linear combination with  coefficients in $\bbN^+$ of two or more elements of $S$.  The latter set is uniquely determined by $S$. Thus $S^{\prime}$ must  be unique as well.
\end{proof}

\begin{lemma}\label{lemma-quasi-antichain-Cox-unique}
If $S_1, S_2\subseteq \Delta^{+}$ are two quasi-antichains such that $H_{S_1}=H_{S_2}$, then $S_1=S_2$.
\end{lemma}
\begin{proof}
On one hand, $S_1$ and  $S_2$ are  two quasi-antichains contained in $(S_1\cup S_2)$.
On the other hand, by Lemma \ref{lemma-S-gen-S},
$$
H_{(S_1\cup S_2)}=H_{S_1\cup S_2}=H_{S_1}\cap H_{S_2}=H_{S_1}=H_{S_2}.
$$
Thus $S_1=S_2$ by  Lemma \ref{lemma-quasi-antichain-Cox}.
\end{proof}

\begin{thm}\label{thm-quasi-antichain-Cox}
The quasi-antichains of $\Delta^{+}$ are in bijection with the elements of $\caL(\emph{\texttt{Cox}})$, regardless of the type of $\frg$.
\end{thm}
\begin{proof}
This  is a combination of Lemmas \ref{lemma-S-gen-S},  \ref{lemma-quasi-antichain-Cox} and \ref{lemma-quasi-antichain-Cox-unique}.
\end{proof}

\section{Characteristic polynomial of \texttt{Shi}($I$): type $A$}

Let $\frg=A_{n-1}$, $n \geq 2$. We choose $\Delta^+=\{e_i-e_j|1\leq i<j\leq n\}$. We will refer to   $e_i-e_j$ simply by $ij$. They span the real vector space $V=\{ (x_1, \cdots, x_n)\in\bbR^n\mid x_1 + \cdots + x_n=0 \}$. The corresponding set of simple roots is
$\Pi=\{ 12, 23, \cdots, (n-1)n  \}$.
Recall from (\ref{H-alpha-k}) that  the arrangements \texttt{Cox} and \texttt{Shi} etc are defined within $V$.
This section is mainly devoted to investigating the characteristic polynomial and the freeness of \texttt{Shi}($I$).

\subsection{Set partitions and quasi-antichains}
Put $[n]:=\{1,2,\cdots, n\}$. We say that $\pi=\{B_1, B_2,\cdots, B_k\}$,  where each $B_i$ is a nonempty subset of $[n]$, is a \emph{partition of}  $[n]$ \emph{into $k$ blocks} if we have the disjoint union
$$
[n]=B_1 \sqcup B_2\sqcup\cdots \sqcup B_k.
$$
The \emph{arc diagram} of $\pi$ is drawn as follows: place the numbers $1, 2, \cdots, n$ on a line and draw an arc between each pair $i<j$ such that
\begin{itemize}
\item[$\bullet$] $i$ and $j$ are in the same block of $\pi$; and
\item[$\bullet$] there is no $i<l<j$ such that $i, l, j$ are in the same block of $\pi$.
\end{itemize}
A partition $\pi$ of $[n]$ is called \emph{nonnesting} if it does not contain arcs $ij$ and $kl$ such that $i<k<l<j$---that is, no arc of $\pi$ ``nests" inside another.
Figure 1 displays the arc diagrams for the partitions $\{\{1, 3\}, \{2\}, \{4, 5\}\}$ and $\{\{1, 5\}, \{2\}, \{3, 4\}\}$ of $[5]$. The first partition is nonnesting, while the second one is nesting since the arc $34$ lies inside the arc $15$.

Let $\Gamma_{n}$ be the staircase Young diagram.
Following Shi \cite{S3}, let us fill the positive roots of $A_{n-1}$ into the boxes of $\Gamma_{n-1}$ as follows: put the root $ij$ into the ($n+1-j$, $i$)-th box $B_{n+1-j, i}$. Figure 2 is an example with $n=5$.

We will \emph{identify} $\Delta^{+}$ with $\Gamma_{n-1}$ in this way, and transfer the partial ordering $\leq$ on $\Delta^{+}$ to a partial ordering on $\Gamma_{n-1}$ accordingly.
We note that for any two positive roots $\alpha$ and $\beta$,
$\alpha\leq\beta$ if and only if the  box of $\beta$ is to the north, or the west, or the northwest of the box of $\alpha$. This describes the poset structure of $\Gamma_{n-1}$. By a bit of abuse of notation, we will refer to this poset simply by $\Gamma_{n-1}$.
We can characterize the quasi-antichains of $\Gamma_{n-1}$ as follows: a subset $S$ of $\Gamma_{n-1}$ is a quasi-antichain if and only if
in each row and  column of $\Lambda_n$, there is at most one box of $S$. Now let us state an easy observation which should be well-known.

\begin{lemma}\label{lemma-quasi-antichain-partition-typeA}
There exists a bijective map from the set of all the quasi-antichains of $\Gamma_{n-1}$ to the set of all partitions of $[n]$. Moreover,  this map sends the quasi-antichains of $\Gamma_{n-1}$ with size $n-k$ to the partitions of $[n]$ with $k$ blocks.
\end{lemma}
\begin{proof}
Given any quasi-antichain $A$, let us draw an arc between the indices $i$ and $j$ for each element $ij\in A$. Then we end up with the arc graph of a partition $\pi(A)$. We map $A$ to $\pi(A)$, which is easily seen to be bijective.
Since a partition has $k$ blocks if and only if its arc diagram has $n-k$ arcs, the second statement follows directly.
\end{proof}
\begin{rmk}
 One sees easily that the map in the above lemma sends the set of all the antichains of $\Gamma_{n-1}$ to the set of all the nonnesting partitions of $[n]$.
\end{rmk}

\begin{figure}[H]
\centering
\scalebox{0.4}{\includegraphics{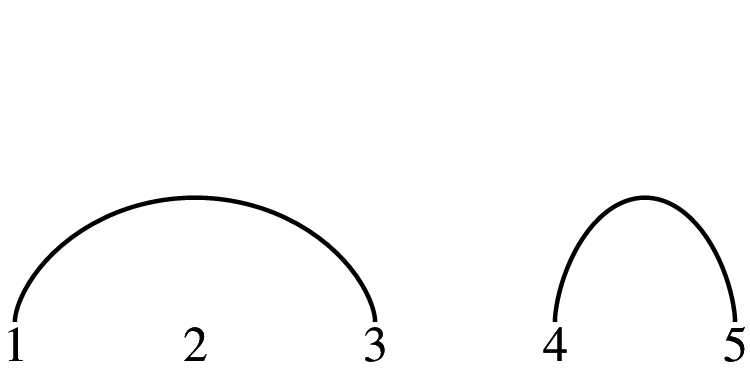}}
\qquad\quad
\scalebox{0.4}{\includegraphics{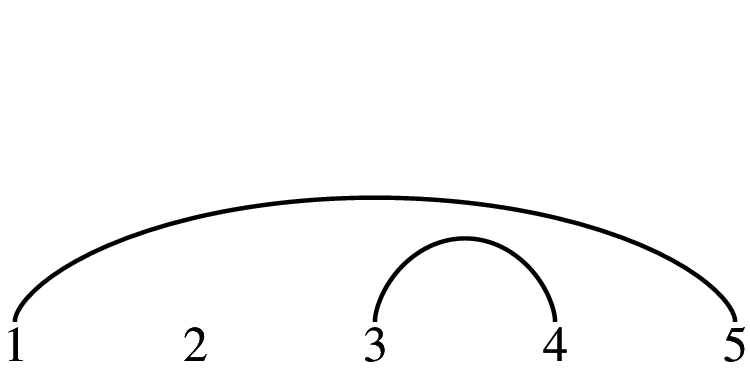}}
\caption{ Two partitions of $[5]$}
\end{figure}

\begin{figure}[H]
\centering
\scalebox{0.35}{\includegraphics{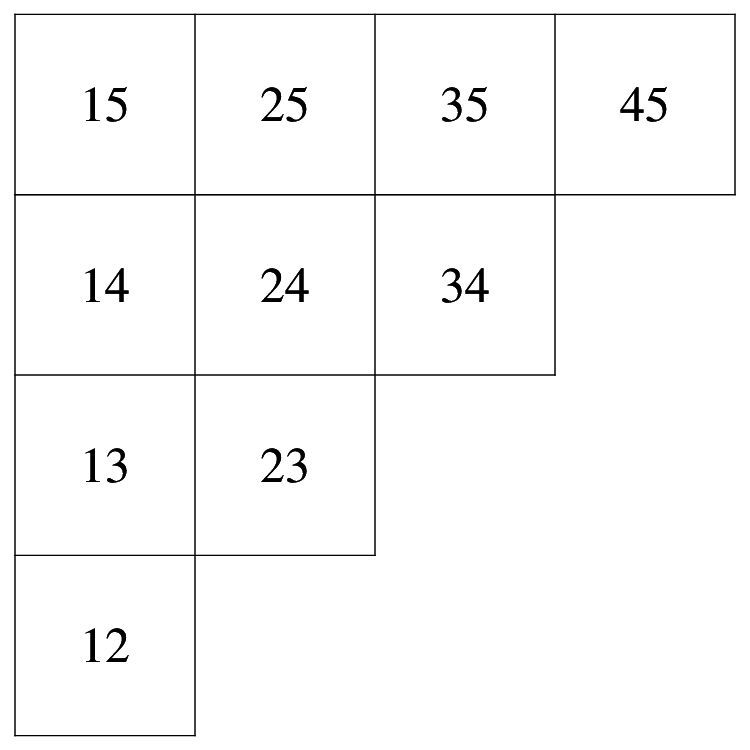}}
\caption{The Young diagram $\Gamma_4$ for $A_4$}
\end{figure}

\subsection{The characteristic polynomial and the freeness of \texttt{Shi}$(I)$}
For any subset $G \subseteq\Gamma_{n-1}$,
recall that
$$\texttt{Shi}(G)=\{x_i-x_j=0, 1\leq i<j\leq n\} \cup \{x_i-x_j=1, ij\in G \}.$$
Recall also that in \cite{AR} a partition $\pi$ of $[n]$ is called a \emph{$G$-partition} if all its arcs are contained in $G$, and the \emph{$G$-Stirling number} $\texttt{Stir}(G, k)$ is defined as the number of $G$-partitions with $k$ blocks.  Thus $\texttt{Stir}(\Gamma_{n-1}, k)=S(n, k)$. Moreover, in view of Lemma \ref{lemma-quasi-antichain-partition-typeA}, we have \begin{equation}\label{Stir-A}
\texttt{Stir}(G, k) = \# \Big\{ \mbox{quasi-antichains of }
\Gamma_{n-1} \mbox{ contained in } G \mbox{ with size } n-k\Big\}.
\end{equation}
This reinterpretation of $\texttt{Stir}(G, k)$ motivates our uniform definition (\ref{Stir-uniform}). Now let us state a result of Armstrong and Rhoades in \cite{AR}, which is deduced by the finite field method described in Theorem \ref{thm-finite-field}.

\begin{thm}\label{thm-AR} \emph{(Theorem 3.2 of \cite{AR})} For any subset $G \subseteq \Gamma_{n-1}$, the characteristic polynomial of $\emph{\texttt{Shi}}(G)$ is given by
$$
\chi(\emph{\texttt{Shi}}(G), t)=\sum_{k=0}^{n-1} (-1)^k \emph{\texttt{Stir}}(G, n-k) (t-k-1)_{n-1-k}.
$$
\end{thm}

\begin{cor}\label{Cor-type-A} Let $p$ be a large prime. For any subset $S\subseteq \Gamma_{n-1}$ with cardinality $0\leq k\leq n-1$, the following are equivalent:
\begin{itemize}
\item[(i)] $S$ is a quasi-antichain of $\Gamma_{n-1}$;
\item[(ii)] in each row and column of $\Gamma_{n-1}$, there is at most one box of $S$;
\item[(iii)] $f(S)$ is nonzero;
\item[(iv)] $f(S)=(p-k-1)!/(p-n)!$.
\end{itemize}
\end{cor}
\begin{proof}
It follows from \S 4.1 and the proof of Theorem \ref{thm-AR}.
\end{proof}

\begin{example}\label{exam-shi-cox-A} When $G$ is the empty set, by Theorem \ref{thm-AR}, we have
$$
\chi(\texttt{Cox}, p)=(p-1)_{n-1}.
$$
When $G=\Gamma_{n-1}$, by Theorem \ref{thm-AR} and $(a)_n=(-1)^n (n-a-1)_n$, we have
\begin{align*}
\chi(\texttt{Shi}, p)
&=\sum_{k=0}^{n-1}(-1)^k S(n, n-k)(p-k-1)_ {n-k-1}\\
&=\sum_{k=1}^{n}(-1)^{n-1}S(n, k) (n-p-1)_{k-1} \\
&=\frac{(-1)^{n-1}}{n-p}\sum_{k=0}^{n}S(n, k)(n-p)_{k}\\
&= (p-n)^{n-1},
\end{align*}
where  the penultimate equality holds since $S(n, 0)=0$ for $n>0$,  and the final equality uses (\ref{Gamman-identity}). \end{example}

\begin{thm}\label{thm-type-A}  Let $\frg$ be $A_{n-1}$ and let $I$ be any subset of $\Pi$ with cardinality $r\geq 1$. Then $\chi(\emph{\texttt{Shi}(I)}, t)$ is independent of the $r$ elements that $I$ contains, and
\begin{equation}\label{exp-type-A}
 \exp(\emph{\texttt{Shi}}(I))=\{(n-r)^{ n-r}, (n-r+1)^1, (n-r+2)^1, \cdots, (n-1)^1\}.
 \end{equation}
In particular, the number of  regions  is
$(n-r+1)^{ n-r-1} n!/(n-r)!$ and the number of bounded regions  is $(n-r-1)^{ n-r-1} (n-2)!/(n-r-2)!$.
Moreover, the cone of \emph{\texttt{Shi}(I)}  is free.
\end{thm}

\begin{proof} Fix any subset $I$ of $\Pi$ with cardinality $r\geq 1$.
As noticed by Righi in \S 5.1 of \cite{R1}, the poset $(C_I, \leq)$ is isomorphic to $\Gamma_{n-1-r}$. See Figure 3 for two examples. Moreover, by Corollary \ref{Cor-type-A}(ii),
 $\texttt{Stir}(C_I, n-k)$ equals to the number of quasi-antichains
of $\Gamma_{n-1-r}$ with size $k$. By Lemma \ref{lemma-quasi-antichain-partition-typeA}, the latter is $S(n-r, n-r-k)$, which is nonzero only if $0\leq k\leq n-r$.
Therefore by Theorem \ref{thm-AR}, we have
$$
\chi(\texttt{Shi}(I), t)=\sum_{k=0}^{n-r} (-1)^k S(n-r, n-r-k)(t-k-1)_{n-k-1}.
$$
This tells us that $\chi(\texttt{Shi}(I), t)$ depends only on $r$.
To deduce its explicit expression, the first way is quoting the
formula (\ref{Gamman-identity}), which gives
$$
\sum_{k=0}^{n-r} S(n-r, k)(t)_k=t^{n-r}.
$$
Then after some elementary calculations similar to those presented in Example \ref{exam-shi-cox-A}, one arrives at  (\ref{exp-type-A}).
Alternatively, since $\chi(\texttt{Shi}(I), t)$ is independent of the $r$ elements that $I$ contains,
 we can focus on the special case that $I_0=\{12, \cdots, r(r+1)\}$. Then
$C_{I_0}=\{ij\mid r+1\leq i<j\leq n \}$,
and (\ref{exp-type-A}) follows from Theorem 2.2 of \cite{At2}.

For the last statement, we use Theorem 4.1 of \cite{At2}. It suffices to rule out the following two possibilities:
\begin{itemize}
\item[(a)] there exists $1\leq i<j<k\leq n$ such that $ij, jk\in C_I$ but $ik\notin C_I$;
\item[(b)] there exists four distinct numbers $i_1<j_1, i_2<j_2\in [n]$ such that $i_1 j_1, i_2j_2$ gives \emph{all} the edges between the vertices  $\{i_1, j_1, i_2, j_2\}$ in $C_I$.
\end{itemize}
Suppose that (a) happens. Then by Lemma \ref{lemma-CI}, the (not necessarily distinct) simple roots $i(i+1), (j-1)j, j(j+1), (k-1)k$ are not contained in $I$. Then $ik\notin C_I$ means that there is a simple root $\alpha\in I$ such that $ik-\alpha$ is a root. Thus, $\alpha$ has to be $i(i+1)$ or $(k-1)k$, which is absurd. Now suppose that (b) happens. Without loss of generality, we assume that $i_1<i_2$. Then there are three cases: (1) $i_1<j_1<i_2<j_2$; (2) $i_1<i_2<j_1<j_2$; (3) $i_1<i_2<j_2<j_1$. Suppose that (3) happens. Then by Lemma \ref{lemma-CI}, the (not necessarily distinct) simple roots $i_1(i_1+1), (j_1-1)j_1, i_2(i_2+1), (j_2-1)j_2$ are not contained in $I$. Thus we would have $i_1 j_2\in C_I$, contradiction. (1) and (2) can be ruled out similarly.
\end{proof}

\begin{figure}[H]
\centering
\scalebox{0.35}{\includegraphics{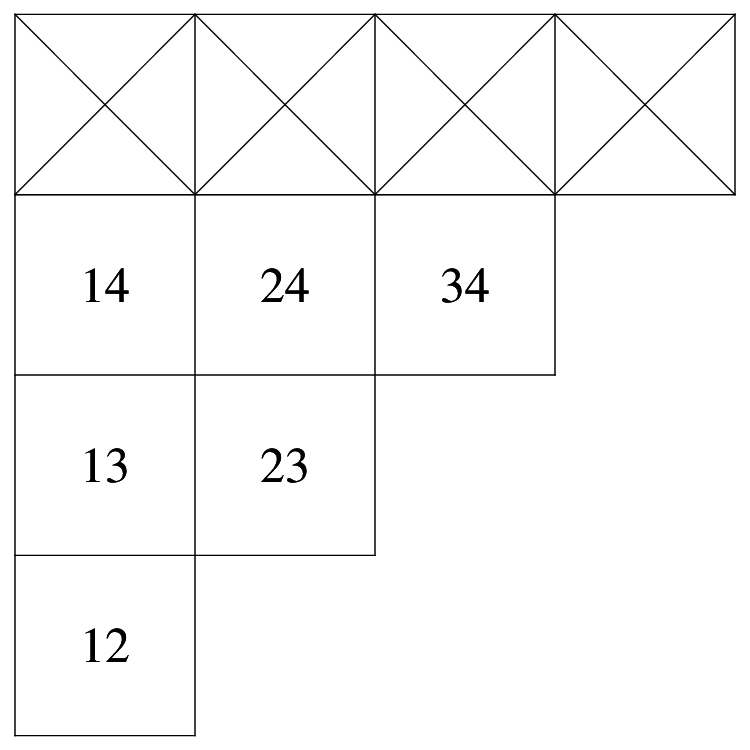}}
\qquad
\scalebox{0.35}{\includegraphics{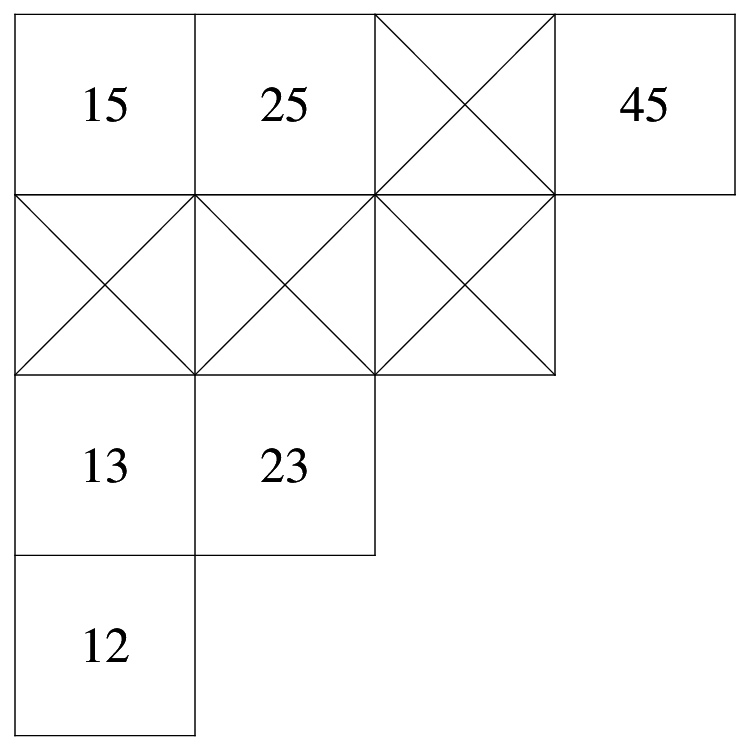}}
\caption{$\frg=A_4$, $I=\{45\}$, $I^{\prime}=\{34\}$}
\end{figure}

\section{Characteristic polynomial of \texttt{Shi}($I$): type $C$}
Let $\frg$ be of type $C_{n}$, $n \geq 2$. We choose  $\Delta^+=\{e_i-e_j\mid 1\leq i<j\leq n\}\cup \{2 e_i\mid 1\leq i\leq n \}$. Then  $\Pi=\{e_1-e_2, e_2-e_3, \cdots,  e_{n-1}-e_n, 2e_n\}$.  They span the real vector space $\mathbb{R}^n$. This section is devoted to proving the following

\begin{thm}\label{thm-type-C}  Let $\frg$ be $C_{n}$ and let $I$ be any subset of $\Pi$ with cardinality $1\leq r\leq n-1$.
Then when $I$ contains $2e_n$, $\chi(\emph{\texttt{Shi}}(I), t)$ is independent of the other $r-1$ simple roots that $I$ contains, and we have
\begin{equation}\label{exp-type-C1}
\exp(\emph{\texttt{Shi}(I)})=\{ (2n-2r+1)^{n-r+1}, (2n-2r+3)^1, (2n-2r+5)^1, \cdots,  (2n-1)^1\}.
\end{equation}
In particular, the number of  regions is
$(2n-2r+2)^{n-r} (2n)!!/(2n-2r)!!$ and
the number of bounded regions is
$(2n-2r)^{n-r}(2n-2)!!/(2n-2r-2)!!$.
Similarly, when $I$ does not contain $2e_n$, $\chi(\emph{\texttt{Shi}}(I), t)$ is independent of the $r$ simple roots that $I$ contains, and we have
\begin{equation}\label{exp-type-C2}
\exp(\emph{\texttt{Shi}(I)})=\{ (2n-2r)^{n-r}, (2n-2r+1)^1, (2n-2r+3)^1, \cdots,  (2n-1)^1\}.
\end{equation}
In particular, the number of  regions is
$(2n-2r+1)^{n-r} (2n)!!/(2n-2r)!!$ and
the number of bounded regions is
$(2n-2r-1)^{n-r} (2n-2)!!/(2n-2r-2)!!$.
\end{thm}

\subsection{Quasi-antichains: characterization and representation}
Following Shi \cite{S3}, let us fill the positive roots of $C_n$ into the staircase Young diagram $\Lambda_n$ as follows: for $1\leq i\leq n$, put $2e_i$ into the $(i, i)$-th box $B_{i, i}$; for $1\leq i <j\leq n$, put $e_i+e_j$ (resp. $e_i-e_j$) into the $(i, j)$-th box $B_{i, j}$ (resp. the $(i, 2n+1-j)$-th box $B_{i, 2n+1-j}$). The $C_4$ case is illustrated in Figure 4.

\begin{figure}[H]
\centering
\scalebox{0.55}{\includegraphics{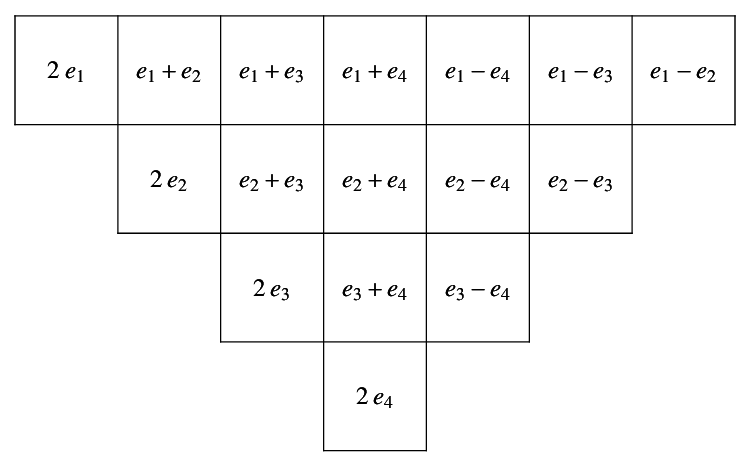}}
\caption{The staircase Young diagram $\Lambda_4$ for $C_4$}
\end{figure}

We will \emph{identify} $\Delta^{+}$ with $\Lambda_n$ in this way, and transfer the partial ordering $\leq$ on $\Delta^{+}$ to a partial ordering on $\Lambda_n$ accordingly.
We note that for any two positive roots $\alpha$ and $\beta$,
$\alpha\leq\beta$ if and only if the  box of $\beta$ is to the north, or the west, or the northwest of the box of $\alpha$. This describes the poset structure of $\Lambda_n$.  By a bit of abuse of notation, we will refer to this poset simply by $\Lambda_n$. For later reference, we denote by $T_{n-1}$ the subdiagram obtained from $\Lambda_n$ by deleting the $n$-th column. The diagram $T_{n-1}$ inherits a partial ordering from $\Lambda_n$.
 Let $L_1=\{2e_i\mid 1\leq i\leq n\}$, and define $L_j$, $2\leq j\leq n-1$, to be the collection of roots on the $j$-th row and the $j$-th column of $\Lambda_n$.
We can characterize the quasi-antichains of $\Lambda_n$ as follows: a subset $S$ of $\Lambda_n$  is a quasi-antichain if and only if
in each row and column of $\Lambda_n$, and in each $L_j$, $1\leq j\leq n-1$, there is at most one box of $S$.

Put $\pm [n]=\{\pm 1, \cdots, \pm n\}$. As in \cite{At1}, we call the elements of $\pm [n]$ the \emph{signed integers} from $1$ to $n$.
Let $S\subseteq\Delta^{+}$ be a quasi-antichain with cardinality $n-k$.
Let us represent $S$ in the following vivid way:
put $n$ boxes on a line and fill $-i, i$ in the $i$-th box for $1\leq i\leq n$. If $2e_i\in S$ (there is at most one such $i$), delete the $i$-th box. If $e_i - e_j\in S$ (resp. $e_i + e_j\in S$), for some $1\leq i<j\leq n$, draw an arc between $i$ and $j$ (resp. $i$ and $-j$) from the below, and draw an arc between $-i$ and $-j$ (resp. $-i$ and $j$) from the above.
Carrying out this process for all the $n-k$ roots in $S$, we call the resulting graph the \emph{signed partition associated to} $S$ and denote it by $\pi^{\pm}(S)$. For example, when $n=5$ and $S=\{e_1-e_2, e_2-e_3\}$, the corresponding signed partition $\pi^{\pm}(S)$ is given in Figure 5.

\begin{figure}[H]
\centering
\scalebox{0.5}{\includegraphics{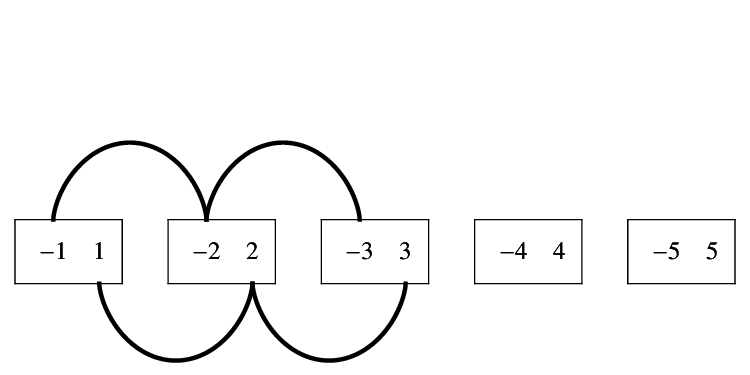}}
\caption{A signed partition for $C_5$}
\end{figure}

We read $\pi^{\pm}(S)$ from the above in the following way: when there is an arc between two signed integers, single them out and join them with an arrow pointing to the right; when there is no arc on a box, single out the positive number in it. For example, reading Figure 5 from the above gives $$-1\rightarrow -2\rightarrow -3, 4, 5.$$
Similarly, we read $\pi^{\pm}(S)$ from the below in the following way: when there is an arc between two signed integers, single them out, swap them, and then join them with an arrow pointing to the right; when there is no arc on a box, single out the negative number in it. For example, reading Figure 5 from the below gives $$3\rightarrow 2\rightarrow 1, -4, -5.$$
We call $3\rightarrow 2\rightarrow 1$ the \emph{negative} of $-1\rightarrow -2\rightarrow -3$, call $-4$ the negative of $4$ etc. Then up to a choice of sign, there are  $3$ \emph{essentially different} ordered parts for the signed partitions in Figure 5. One sees easily that similar things hold in general. Namely, no matter $S$ contains a root of the form $2e_i$ or not,
reading $\pi^{\pm}(S)$ from the above and the below always gives $2k$ ordered parts; moreover, the negative of each part occurs exactly once, and up to a sign, there are  $k$ essentially different ordered parts of $\pi^{\pm}(S)$. As we shall see in the next subsection, $\pi^{\pm}(S)$ is introduced to facilitate the counting of $f(S)$ in the way of Athanasiadis \cite{At1}.

Let $S(\Lambda_n, n-k)$ (resp. $S(T_n, n-k)$) be the number of quasi-antichains of $\Lambda_n$ (resp. $T_n$) with size $k$, which is nonzero only if $0\leq k\leq n$.
The following lemma gives analogs of (\ref{stir-recur-A}) and (\ref{Gamman-identity}).
\begin{lemma}\label{lemma-Tn-Lambdan}
\emph{(i)} For $\Lambda_n$, we have the recurrence
\begin{equation}\label{Lambdan-recurrence}
S(\Lambda_n, k)=S(\Lambda_{n-1}, k-1)+(2k+1) S(\Lambda_{n-1}, k)
\end{equation}
and the identity
\begin{equation}\label{Lambdan-identity}
\sum_{k=0}^{n} S(\Lambda_{n}, k) 2^k (t)_{k}=(2t+1)^{n}.
\end{equation}

\emph{(ii)} For $T_n$, we have the recurrence
\begin{equation}\label{Tn-recurrence}
S(T_n, k)=S(T_{n-1}, k-1)+(2k+2)S(T_{n-1}, k)
\end{equation}
and the identity
\begin{equation}\label{Tn-identity}
\sum_{k=0}^{n} S(T_{n}, k) 2^k (t-1)_{k}=(2t)^{n}.
\end{equation}

\end{lemma}
\begin{proof}
We only provide the proof for the $\Lambda_n$ case. The $T_n$ case is entirely similar.
 Since the subdiagram obtained from $\Lambda_n$ by deleting the first row is isomorphic to $\Lambda_{n-1}$ as posets, we denote it by $\Lambda_{n-1}$. A quasi-antichain $S$ of $\Lambda_n$  with size $k$ can be formed in two ways: no element of $S$ is chosen in the first row of $\Lambda_n$; one element of $S$ in chosen in the first row of $\Lambda_n$. The first number is $S(\Lambda_{n-1}, n-1-k)$. To figure out the second number, we observe that for any box $B_{i, j}$ of $\Lambda_n$, where $i\geq 2$, thanks to the existence of $L_1, \cdots, L_{n-1}$, there are \emph{always} two elements in the first row of $\Lambda_n$ which can not be chosen to form a quasi-antichain together with $B_{i,j}$. Suppose that we have an arbitrary quasi-antichain $S^{\prime}$ with  size $k-1$ of $\Lambda_{n-1}$ at hand, and we want to add one element $B_{1, j}$ from the first row such that $S^{\prime}\cup\{B_{1, j}\}$ is a quasi-antichain of $\Lambda_n$.  Now for each element in $S^{\prime}$, by the previous observation, there are two candidates in the first row that should be avoided. Moreover, there is no overlapping among these candidates since $S^{\prime}$ is a quasi-antichain. Thus there are exactly $2n+1-2k$ allowable choices for $B_{1, j}$. Therefore, the second number is  $(2n+1-2k)S(\Lambda_{n-1}, n-k)$, and we have
\begin{equation}\label{Lambdan-recurrence-re}
S(\Lambda_n, n-k)=S(\Lambda_{n-1}, n-1-k)+(2n-2k+1) S(\Lambda_{n-1}, n-k),
\end{equation}
(\ref{Lambdan-recurrence}) is a reformulation of (\ref{Lambdan-recurrence-re}).
Finally, (\ref{Lambdan-identity}) follows from  (\ref{Lambdan-recurrence}) and an induction argument.
\end{proof}

\subsection{The formula for $f(S)$}
Let $S$ be a quasi-antichain of $\Delta^{+}$ with size $n-k$. By definition, $f(S)$ equals to the number of $n$-tuples $(x_1, \cdots, x_n)\in\mathbb{F}_p^n$ satisfying  $(H_{\alpha, 1})_p$, for all $\alpha\in S$,  as well as
$x_i\neq 0$, for$1\leq i\leq n$;
$x_i \pm x_j\neq 0$, for $1\leq i<j \leq n$.
Let us adopt the way of Athanasiadis \cite{At1} to count $f(S)$. Namely, we think of each $n$-tuple in $\mathbb{F}_p^n$ as a map from $\pm [n]$ to $\mathbb{F}_p$, sending $i$ to the class $x_i\in\mathbb{F}_p$, and $-i$ to the class $-x_i$. We think of the elements of $\mathbb{F}_p$ as boxes arranged and labeled cyclically with the classes mod $p$. The top box is labeled with the zero class, the clockwise next box is labeled with the class 1 mod $p$ etc. Then the $n$-tuples in $\mathbb{F}_p^n$ become \emph{placements} of the signed integers from $1$ to $n$ in the $p$ boxes, and $f(S)$ counts the number of those satisfying $(H_{\alpha, 1})_p$,  $\forall\alpha\in S$, as well as the following
 \begin{itemize}
\item[(a)] when a signed integer is placed in the class $a$, its negative is placed in the class $-a$;
\item[(b)] the zero class is always empty; and
\item[(c)] distinct signed integers are placed in distinct classes.
 \end{itemize}

\begin{lemma}\label{lemma-fS-type-C}
Let $S$ be a quasi-antichain of $\Delta^{+}$ with size $n-k$. Let $p$ be a large prime, then
$$f(S)=\frac{(p-2n+2k-1)!!}{(p-2n-1)!!}.$$
\end{lemma}
\begin{proof}
For convenience, we call the classes from $1$ to $(p-1)/2$ (resp., from $(p+1)/2$ to $p-1$), both included, the \emph{right half} (resp. \emph{left half}) of the circle.
By the definition of $\pi^{\pm}(S)$, $f(S)$ equals to the number of ways of placing the $2k$ ordered parts of $\pi^{\pm}(S)$ on the circle such that  (a-c)  are met. Namely, we should place them  on the circle such that
\begin{itemize}
\item[$\bullet$] each ordered part is consecutive and clockwise; and there is no overlapping;
\item[$\bullet$] each ordered part is either entirely  on the right half or entirely on the left half; if it is on the right half, then its negative is  on the left half according to (a).
\end{itemize}
Thus we can focus on what happens only on the right half, and it boils down to put the $k$ essentially different ordered parts there, with each part a choice of sign.

Note that there are  two types of $S$.
The  first type is that $S$ contains $2e_i$ for some $i\in [n]$. Then $-i$ is sent to the $(p-1)/2$ class, and $i$ to the $(p+1)/2$ class.
 Since there are $(p-1)/2-n$ empty boxes on the right half circle in total,  the allowable ways to place the $k$ essentially different parts there is
$f(S)=2^k \prod_{j=1}^{k}((p-1)/2-n+j)$, which equals to
$(p-2n+2k-1)!!/(p-2n-1)!!$, as desired.

Now suppose that $S$ does not contain any $2e_i$. Then there are two cases:
\begin{itemize}
\item[(i)] the class $(p-1)/2$ is empty. Therefore, to place the $k$ essentially different parts in the right half circle, we should always avoid the class $(p-1)/2$.  Since besides the class $(p-1)/2$, there are $(p-1)/2-n-1$ empty boxes there in total, the allowable ways are easily seen to be
    $2^k\prod_{j=1}^{k}((p-1)/2-n-1+j)$.
\item[(ii)] the  class $(p-1)/2$ is filled. Then there must be an ordered part of $\pi^{\pm}(S)$ which is placed entirely on the right half circle such that it ends with the class $(p-1)/2$. We have $2k$ different ways to choose this ordered part.
Then we have to place the remaining $k-1$ essentially different ordered parts in the remaining classes of the right half. Since there is a  choice of sign for each of them, and
there are $(p-1)/2-n$ empty boxes on the right half in total, the allowable ways are $(2k)2^{k-1}\prod_{j=1}^{k-1}((p-1)/2-n+j)$.
\end{itemize}
Summing the numbers in (i) and (ii) gives the \emph{same} value of $f(S)$ as in the first type. This finishes the proof.
\end{proof}

\begin{cor}\label{Cor-type-C} Let $p$ be a large prime. For any subset $S\subseteq\Lambda_n$ with cardinality $0\leq k\leq n$, the following are equivalent:
\begin{itemize}
\item[(i)] $S$ is a quasi-antichain of $\Lambda_n$;
\item[(ii)] in each row and column of $\Lambda_n$, and in each $L_j$, $1\leq j\leq n-1$, there is at most one box of $S$;
\item[(iii)] $f(S)$ is nonzero;
\item[(iv)] $f(S)=(p-2k-1)!!/(p-2n-1)!!$.
\end{itemize}
\end{cor}
\begin{proof}
It follows from \S 5.1 and Lemma \ref{lemma-fS-type-C}.
\end{proof}

\begin{thm}\label{thm-ShiG-type-C}
Let $\frg$ be $C_n$, $n\geq 2$.
Then for any subset $G\subseteq\Lambda_n$, the characteristic polynomial of $\emph{\texttt{Shi}}(G)$ is given by
$$
\chi(\emph{\texttt{Shi}}(G), t)=\sum_{k=0}^{n} (-1)^k \emph{\texttt{Stir}}(G, n-k)2^{n-k}(\frac{t-1}{2}-k)_{n-k}.
$$
\end{thm}
\begin{proof}
This follows from Lemma \ref{lemma-sum-quasi-antichain} and Lemma \ref{lemma-fS-type-C}.
\end{proof}
\begin{rmk}\label{rmk-type-B}
Since the positive roots for $B_n$ are
$\{e_i-e_j\mid 1\leq i<j\leq n\}\cup \{ e_i\mid 1\leq i\leq n \}$,
one sees immediately that Lemma \ref{lemma-fS-type-C}, thus the theorem above, holds for $B_n$.
\end{rmk}

\begin{example}\label{exam-shi-cox-C} When $G$ is the empty set, by Theorem \ref{thm-ShiG-type-C}, we have
$$
\chi(\texttt{Cox}, p)=2^{n}(\frac{p-1}{2})_{n}.
$$
When $G=\Lambda_{n}$,  by Theorem \ref{thm-ShiG-type-C},
$$
\chi(\texttt{Shi}, p)=\sum_{k=0}^{n}(-1)^k S(\Lambda_n, n-k)2^{n-k}(\frac{p-1}{2}-k)_{n-k}.
$$
Then using (\ref{Lambdan-identity}) and  a sequence of steps analogous to those in Example \ref{exam-shi-cox-A}, we have
$$\chi(\texttt{Shi}, p)=(p-2n)^n.$$
\end{example}

\subsection{Proof of Theorem \ref{thm-type-C} when $I$ contains $2e_n$}
Fix any subset $I\subseteq\Pi$ with cardinality $1\leq r\leq n-1$ which  contains $2e_n$. As noticed by Righi in \S5.2 of \cite{R1}, the poset
$(C_I, \leq)$ is isomorphic to $T_{n-r}$. See Figure 6 for two examples. Moreover, by Corollary \ref{Cor-type-C}(ii),  $\texttt{Stir}(C_I, n-k)$ equals to the number of quasi-antichains of $T_{n-r}$ with size $k$. Note that the latter is $S(T_{n-r}, n-r-k)$, which is nonzero only if $0\leq k\leq n-r$. Now by
Theorem \ref{thm-ShiG-type-C},
$$
\chi(\texttt{Shi}(I), t)=\sum_{k=0}^{n-r} (-1)^k S(T_{n-r}, n-r-k)(t-2k-1)_{n-k}.
$$
This tells us that $\chi(\texttt{Shi}(I), t)$ is independent of the other $r-1$ simple roots that $I$ contains.
To arrive at (\ref{exp-type-C1}), after some elementary calculations similar to those in Example \ref{exam-shi-cox-C}, it boils down to show
\begin{equation}
\sum_{k=0}^{n-r} S(T_{n-r}, k) 2^k(t-1)_{k}=(2t)^{n-r},
\end{equation}
which is an easy consequence of Lemma \ref{lemma-Tn-Lambdan}(ii).

\subsection{Proof of Theorem \ref{thm-type-C} when $I$ does not contain $2e_n$}
Fix any subset $I\subseteq\Pi$ with cardinality $1\leq r\leq n-1$ which does not contain $2e_n$. This case is similar to the previous one. The only difference is that that now the poset
$(C_I, \leq)$ is  isomorphic to $\Lambda_{n-r}$.  Then replacing $T_{n-r}$ in the previous case by $\Lambda_{n-r}$ and quoting Lemma \ref{lemma-Tn-Lambdan}(i) instead finish the proof.
\hfill\qed

\begin{figure}[H]
\centering
\scalebox{0.55}{\includegraphics{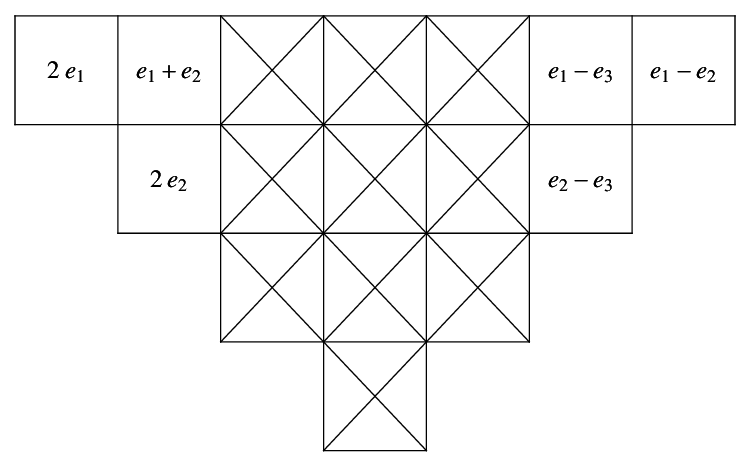}}
\qquad
\scalebox{0.55}{\includegraphics{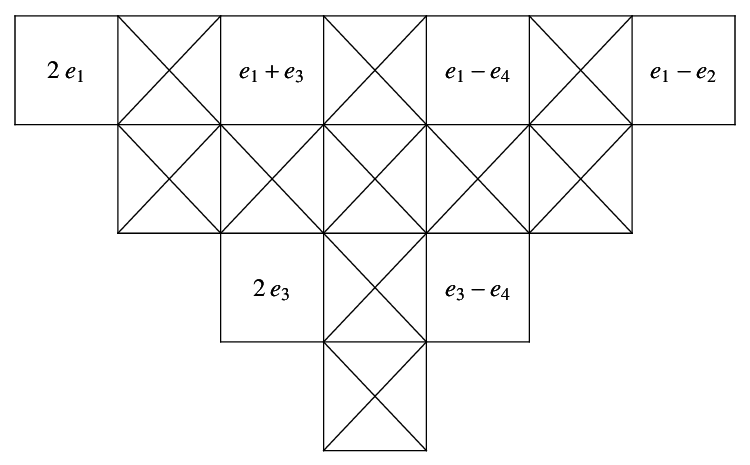}}
\caption{$\frg=C_4$, $I=\{e_3-e_4, 2e_4\}$, $I^{\prime}=\{e_2-e_3, 2e_4\}$}
\end{figure}

\section{Characteristic polynomial of $\texttt{Shi}(G)$: type $D$}

Let $\frg=D_n$, $n\geq 4$. We choose $\Delta^+=\{e_i\pm e_j \mid 1\leq i<j \leq n\}$, then  $\Pi=\{e_1-e_2, e_2-e_3, \cdots,  e_{n-1}-e_n, e_{n-1}+e_n\}$.
They span the real vector space $\mathbb{R}^n$. In this section, we will give a formula for the characteristic polynomial of $\texttt{Shi}(G)$, where $G$ is an arbitrary subset of $\Delta^{+}$. As suggested by Lemma \ref{lemma-sum-quasi-antichain}, this can be achieved by finding a formula for $f(S)$, where $S$ is an arbitrary quasi-antichain of $\Delta^+$ contained in $G$ with size $n-k$.

\begin{example}\label{exam-D} Let us consider $D_4$. Then
\begin{itemize}
\item[$\bullet$] $f(S)=(p-4)(p-5)$ when $S=\{e_1+e_3, e_2-e_3\}$;
\item[$\bullet$] $f(S)=(p-3)(p-5)$ when $S=\{e_1+e_2, e_3+e_4\}$.
\end{itemize}
\end{example}

Thus $f(S)$ \emph{no longer} depends only on $|S|$ for $D_n$.
However, by a more careful analysis, we can still obtain a formula for $f(S)$.
Indeed,
similar to the $C_n$ case, we associate to $S$ a signed partition $\pi^{\pm} (S)$. Reading $\pi^{\pm} (S)$ from the above and the below gives $2k$ ordered parts, where the negative of each part occurs exactly once. Thus up to a sign, there are $k$ essentially different ordered parts. Let $n_1(S)$ be one half of the number of ordered parts with length $1$ in $\pi^{\pm}(S)$. It is necessarily a nonnegative integer.
By definition, $f(S)$ equals to the number of $n$-tuples $(x_1, \cdots, x_n)\in\mathbb{F}_p^n$ satisfying $(H_{\alpha, 1})_p$, for all $\alpha\in S$, as well as $x_i\pm x_j \neq 0$, for $1\leq i < j\leq n$.
Again we adopt the way of Athanasiadis to count \cite{At1}. Then the $n$-tuples in $\mathbb{F}_p^n$ become\emph{ placements} of the signed integers from $1$ to $n$ in the $p$ boxes, and $f(S)$ counts the number of those satisfying $(H_{\alpha, 1})_p$, $ \forall \alpha\in S$, as well as the following
 \begin{itemize}
\item[(a)] when a signed integer is placed in the class $a$, its negative is placed in the class $-a$;
 \item[(b)] there is at most one signed integer placed in each nonzero class;
 \item[(c)] there exists at most one $i\in [n]$ such that both $i$ and $-i$ are placed in the zero class.
 \end{itemize}

\begin{lemma}\label{lemma-fS-type-D}
Let $S$ be a quasi-antichain of $\Delta^{+}$ with size $n-k$. Let $p$ be a large prime, then
$$f(S)=
\frac{(p-2n+2k+1)!!}{(p-2n+1)!!}
-n_1(S) \frac{(p-2n+2k-1)!!}{(p-2n+1)!!},$$
where $n_1(S)$ is one half of the number of ordered parts with length $1$ in $\pi^{\pm}(S)$.
\end{lemma}
\begin{proof}
For convenience, we call the zero class combined with the right (resp. left) half the \emph{extended} right (resp. left) half.
By the definition of $\pi^{\pm}(S)$, $f(S)$ equals to the number of placements of the $2k$ ordered parts of $\pi^{\pm}(S)$ on the circle such that  (a-c) are met.
Namely, we should place them on the circle such that
\begin{itemize}
\item[$\bullet$] each ordered part is consecutive and clockwise; and except for in the zero class, there is no overlapping;
\item[$\bullet$] each ordered part is  either entirely  on the extended right half or entirely on the extended left half; if it is placed on the extended right half, then its negative should be placed on the extended left half according to (a).
\end{itemize}
Thus we can focus on what happens only on the extended right half, and it boils down to put the $k$ essentially different ordered parts there, with each part a choice of sign. To deduce an explicit expression for $f(S)$, we note that there are two cases:
\begin{itemize}
\item[(i)] the zero class is empty. Then the number of allowable placements is already counted by Lemma \ref{lemma-fS-type-C}. Namely, it is $(p-2n+2k-1)!!/(p-2n-1)!!$.

\item[(ii)] the zero class is filled. In such a case, if an ordered part with length $1$ is placed in the zero class, then its negative must be placed there as well; while for an ordered part with length greater than $1$, we can place it (or its negative) entirely on the extended right half  starting with the zero class. Thus, there are $2k- n_1(S)$ ways to use an ordered part to fill the zero class. Then it boils down to place the remaining $k-1$ essentially different ordered parts on the right half. Since there are $(p-1)/2-n+1$ empty boxes there in total, similar to the second type of Lemma \ref{lemma-fS-type-C}, one  can count that the latter number is $(p-2n+2k-1)!!/(p-2n+1)!!$. Thus the total number is $(2k-n_1(S))(p-2n+2k-1)!!/(p-2n+1)!!$
\end{itemize}
Summing up the numbers in (i) and (ii) gives the desired expression for $f(S)$.
\end{proof}

\begin{cor}\label{Cor-type-D}
 Let $\frg$ be $D_n$, $n\geq 4$.
 Let $p$ be a large prime.  Let $S\subseteq\Delta^{+}$ be any subset with cardinality $0\leq k\leq n$. Then $S$ is a quasi-antichain of $\Delta^{+}$ if and only if  $f(S)\neq 0$.
\end{cor}

Let $G$ be any subset of $\Delta^{+}$.
We  put
\begin{equation}
\texttt{Stir}_1(G, k)=\sum_S n_1(S),
\end{equation}
where $S$ runs over all the quasi-antichains of $\Delta^+$ contained in $G$ with size $n-k$.

\begin{thm}\label{thm-ShiG-type-D}
Let $\frg$ be $D_n$, $n\geq 4$.
Then for any subset $G\subseteq\Delta^+$, the characteristic polynomial of $\emph{\texttt{Shi}}(G)$ is given by
$$
\chi(\emph{\texttt{Shi}}(G), t)=\sum_{k=0}^{n} (-1)^k
\Big\{
\emph{\texttt{Stir}}(G, n-k) (t-2k+1)
-\emph{\texttt{Stir}}_{1}(G, n-k)
\Big\} 2^{n-1-k}(\frac{t-1}{2}-k)_{n-1-k}.
$$
\end{thm}
\begin{proof}
This follows from Lemma \ref{lemma-sum-quasi-antichain} and Lemma \ref{lemma-fS-type-D}.
\end{proof}

\centerline{\scshape Acknowledgements} The research is supported by
the National Natural Science Foundation of China grant 11201261. In the
process of carrying out this work, the author had more than a few
conversations with S.-J.~Wang about hyperplane arrangements. I
would like to thank him heartily for sharing knowledge with me.
Two anonymous referees gave me many suggestions. In particular, the entire subsection 3.4 was motivated by a question of one referee, and the introduction was reorganized according to the suggestion of another referee. I thank them sincerely. Finally, the author is deeply indebted to the editors.

\end{document}